\documentclass{amsart}
\usepackage{amsmath,amsaddr,fancyhdr,amssymb,amsthm,geometry,enumerate, graphicx,  ,amsfonts,mathrsfs,setspace }
\usepackage{amsmath}
\usepackage{amsfonts}
\usepackage{amssymb}
\usepackage{graphicx}
\usepackage{hyperref}
\hypersetup{colorlinks,%
citecolor= red,%
linkcolor=black,%
urlcolor=black,%
pdftex}
\theoremstyle{plain} 
\newtheorem{thm}{Theorem}[section]

\newtheorem{prop}[thm]{Proposition}

\theoremstyle{definition} 
\newtheorem{Def}[thm]{Definition} 
\newtheorem{lem}[thm]{Lemma}

 \newtheorem*{conc}{Conclusion}
 
\theoremstyle{remark} 
\newtheorem*{rmk}{Remark}

\makeatletter \renewenvironment{proof}[1][\proofname] {\par\pushQED{\qed}\normalfont\topsep6\p@\@plus6\p@\relax\trivlist\item[\hskip\labelsep\bfseries#1\@addpunct{.}]\ignorespaces}{\popQED\endtrivlist\@endpefalse} \makeatother

\begin{document}
\title[A study of topological structures on equi-continuous mappings]{A study of topological structures on equi-continuous mappings}
\author{Ankit Gupta$^1$ and Ratna Dev Sarma$^2$}\address{$^1$Department of Mathematics, University of Delhi, Delhi 110007, India. \\ Email : ankitsince1988@yahoo.co.in
\newline $^2$Department of Mathematics, Rajdhani College $($University of Delhi$)$, Delhi 110015, India.\\ Email : ratna\_sarma@yahoo.com}

\thanks{Corresponding author $:$ Ankit Gupta\\
This paper was prepared during a sabbatical leave of the second author.}
\date{}

\begin{abstract}
Function space topologies are developed for $EC(Y,Z)$, the class of equi-continuous mappings from a topological space $Y$ to a uniform space $Z$. Properties such as splittingness, admissibility etc. are defined for such spaces. The net theoretic investigations are carried out to provide characterizations of splittingness and admissibility of function spaces on $EC(Y,Z)$. The open-entourage topology and point-transitive-entourage topology are shown to be admissible and splitting respectively. Dual topologies are defined. A topology on $EC(Y,Z)$ is found to be admissible (resp. splitting) if and only if its dual is so.
\end{abstract}
\maketitle
\thispagestyle{empty}
\section{Introduction}
Investigations of topological aspects of the collections of continuous mappings from a topological space $Y$ to another topological space $Z$ has been an area of active research in topology. Intrinsic properties of  function space topologies have also been investigated in depth by several researchers. The relationship between convergence and topologies of $\mathcal{C}(X,\mathbb{R})$ and that of the hyperspaces $\mathcal{C}(X,\$)$ of open subsets of $X$ has been studied in \cite{u}. Dual topologies for function space topologies and existence of a greatest splitting topology have been investigated in \cite{g04} and \cite{g08} respectively. Conditions under which compact-open, Isbell or natural topologies etc. on $\mathcal{C}(X, \mathbb{R})$ may coincide have been explored in \cite{J}.  In the recent years, several research papers have come up dealing with certain particular as well as some more general cases of this study. For example, for the particular case $Z = \mathbb{R}$, bounded-open topology and pseudo-compact-open topologies are discussed in \cite{B} and \cite{BB}. In \cite{j} and \cite{jj}, function space topologies arising from strong uniform continuity have been studied. On the other hand in \cite{A} and \cite{AA}, topologies on $Y$ and $Z$ are replaced by fuzzy topologies, which provide a more general set up for topological properties. Similarly, function space topologies for generalized topological spaces have been discussed in \cite{ank}.  In this present paper, we investigate the same for equi-continuous mappings from $Y$ to $Z$, where $Y$ has a topology while $Z$ is equipped with a uniformity. With the help of examples, we have shown that several such topologies do exist really on $EC(Y,Z)$, the collection of equi-continuous mappings from $Y$ to $Z$. As the uniform spaces are positioned between the metric spaces and the topological spaces, there is a tendency to discount their investigations as particular cases of topology. However, through our study, we have shown here that uniform structures and in particular, the equi-continuous mappings need not to be studied from that point of view. Rather the inherent aesthetics and intricacies, arising out of uniformities are best revealed, when the related notions are studied directly, not as by product of topology. In fact, we have also introduced function space topology for the family of $_pEC(Y,Z)$ of pseudo-dislocated equi-continuous mappings. In this case, $Z$ has pseudo-dislocated uniformity, which unlike uniformity, does not generate any topology.
%
%In the section $3$, we introduce a topology over the class of pseudo-dislocated equi-continuous mappings $_pEC(Y,Z)$, which is a weaker form of equi-continuous mappings. But these types of equi-continuous mappings can not be studied as a by product of topology as they do not provide the class of continuous mappings in their topological counter-part.

We have introduced admissibility and splittingness  for $EC(Y,Z)-$ two important features for any function space topology. Using net-theory, we have developed the concept of equi-continuous convergence of nets of equi-continuous functions. Splittingness and admissibility are characterized using the notion of equi-continuous convergence.
These characterizations are used to prove that open-entourage topology on $EC(Y,Z)$ is admissible whereas point-transitive-entourage topology is splitting. In the last section, we have introduced the concept of dual topology on $\mathcal{O}_Z(Y)$, the collection of open sets of $Y$ obtained in relation to the equi-continuous mappings. Interesting relationships are observed between the topologies on $\mathcal{O}_Z(Y)$ and that of $EC(Y,Z)$. For example, a topology on $EC(Y,Z)$ is admissible (resp. splitting) if and only if its dual on $\mathcal{O}_Z(Y)$ is admissible (resp. splitting). Similarly, a topology on $\mathcal{O}_Z(Y)$ is admissible (resp. splitting) if and only if its dual on $EC(Y,Z)$ is so.
\section{Equi-Continuity, Pseudo-dislocated equi-continuity and Convergence}
In this section, we develop the net convergence criterion for equi-continuousas well as pseudo-dislocated-equi-continuity mappings.
\begin{Def}\label{Def:1}
A \textit{uniform structure} or \textit{uniformity} on a non-empty set $X$ is a family $\mathcal{U}$ of subsets of $X \times X$ satisfying following properties:
\begin{enumerate}[$(\mathcal{U}_1)$]
\item if $U \in \mathcal{U}$, then $\Delta X \in U$;\\
here $\Delta X= \{(x,x)\in X \times X$ for all $ x\in X\}$;
\item if $U \in \mathcal{U}$, then $U^{-1} \in \mathcal{U}$;\\
here, $U^{-1}$ is called \textit{inverse relation}  of $U$ and defined as 
\begin{center}
$U^{-1}=\{(x,y)\in X\times X\,|\,(y,x)\in U\}$
\end{center}
\item if $U \in \mathcal{U}$, then there exists some $V \in \mathcal{U}$ such that $V \circ V \subseteq U$;\\
here the composition $U \circ V =\{(x,z)\in X \times X\,|\, $ for some $y \in X$, $(x,y)\in V$ and $(y,z)\in U\}$.
\item if $U,V \in \mathcal{U}$, then $U \cap V \in \mathcal{U}$;
\item if $U \in \mathcal{U}$ and $U\subseteq V \subseteq X \times X$, then $V \in \mathcal{U}$.
\end{enumerate}
\end{Def}
The pair $(X, \mathcal{U})$ is a \textit{uniform space}\index{uniform space} and the members of $\mathcal{U}$ are called \textit{entourages}.
\begin{Def}
Let $(X, \tau)$ and $(Y, \mathcal{U})$  be a topological space and a uniform space respectively. A function $f\,:(X,\tau) \to (Y, \mathcal{U})$  is said to be \textit{equi-continuous} at $x \in X$, if for each entourage $U \in \mathcal{U}$, there exists an open neighbourhood $V$ of $x$ such that $f(V) \subseteq U[f(x)]$, where $U[f(x)] = \{y \in Y\,|\,((f(x),y))\in U\}$.
\end{Def}
If $f$ is equi-continuous for all $x \in X$, then $f$ is called \textit{equi-continuous}.\\
The collection of all equi-continuous functions from $X$ to $Y$ is denoted by $EC(X,Y)$  respectively.

\begin{Def}
Let $\{x_n\}_{n \in D}$ be a net in a uniform space $(X, \mathcal{U})$ . Then $\{x_n\}_{n \in D}$ is said to be \textit{convergent} to $x\in X$, if for each entourage $U \in \mathcal{U}$, there exists an $m \in D$, such that $(x, x_n)\in U$ for all $n \geq m$.
\end{Def}
In our next theorem, we provide the net convergence criteria for equi-continuous functions.
\begin{thm}
Let $(X, \tau)$ and $(Y, \mathcal{U})$ be a topological space and a uniform space respectively.  Then a function $f\,:(X,\tau) \to (Y, \mathcal{U})$ is equi-continuous at $x \in X$ if and only if whenever a net $\{x_n\}_{n \in D}$ converges to $x$ in $X$, its  image net $\{f(x_n)\}_{n \in D}$ converges to $f(x)$ in $Y$.
\end{thm}
\begin{proof}
Let $\{x_n\}_{n \in D}$ be any convergent net in $X$, which converges to $x \in X$ and let $f\,:X \to Y$ be  equi-continuous at $x \in X$. We have to show that the net $\{f(x_n)\}_{n \in D}$ converges to $f(x)$ in $Y$. Let $U \in \mathcal{U}$ be any entourage. Since $f$ is equi-continuous at $x \in X$, therefore there exists an open neighbourhood $V$ of $x$ such that $f(V) \subseteq U[f(x)]$. Since the net $\{x_n\}_{n \in D}$ converges to $x$,  $x_n \in V$ eventually. Hence $f(x_n)\in f(V)\subseteq U[f(x)]$ eventually which implies that $(f(x),f(x_n)) \in U$ eventually. Therefore the image net $\{f(x_n)\}_{n \in D}$ converges to $f(x)$ in $Y$.\\
Conversely, let the hypothesis hold. Let if possible $f$ be not equi-continuous at $x \in X$. Then there exists an entourage $U \in \mathcal{U}$ such that there is no open neighbourhood $V$ of $x \in X$ such that $f(V) \subseteq U[f(x)]$. That is, for each open neighbourhood $V$ of $x$, there exists some $x_V \in V$ such that $f(x_V)\notin U[f(x)]$. Let $D$ be a collection of all open neighbourhoods of $x$. Then $(D,\ge)$ is a directed set under the inverse set inclusion $\ge$, that is, $V \ge U$ if $V \subseteq U$. Then $\{x_V\}_{V \in D}$ is a net in $X$ which converges to $x$. But the image net $\{f(x_V)\}_{V \in D}$ does not converge to $f(x)$, because for $U \in \mathcal{U}$, we have $f(x_V)\notin U[f(x)]$ for all $V \in D $. Thus we get a contradiction. Therefore $f$ is equi-continuous at $x \in X$.   
\end{proof}
Next we provide few results regarding the pseudo-dislocated uniform space and pseudo-dislocated equi-continuous mappings. The importance of these spaces lies in the fact that they do not generate any topology like the uniform spaces do.
\begin{Def}\cite{kumari16}
A \textit{pseudo-dislocated uniformity} on a non-empty set $X$ associated with a subset $A$ of $X$ is a family $\mathcal{U}_A$ of subsets of $X \times X$ which satisfies $(\mathcal{U}_2), (\mathcal{U}_3), (\mathcal{U}_4), (\mathcal{U}_5)$ of Definition \ref{Def:1} together with the following property:
($\mathcal{U}_1')\;$ Every member of $\mathcal{U}_A$ contains $\Delta_A = \{(x,x)\,|\,x \in A\}$. 
\end{Def}
The pair $(X, \mathcal{U}_A)$ is called \textit{pseudo-dislocated uniform space}.
\begin{Def}
Let $(X, \tau)$ and $(Y, \mathcal{U}_A)$  be a topological space and a pseudo-dislocated uniform space respectively. A function $f\,:(X,\tau) \to (Y, \mathcal{U}_A)$  is said to be \textit{pseudo-dislocated equi-continuous} at $x \in X$, if for each entourage $U \in \mathcal{U}_A$, there exists an open neighbourhood $V$ of $x$ such that $f(V) \subseteq U[f(x)]$, where $U[f(x)] = \{y \in Y\,|\,((f(x),y))\in U\}$.
\end{Def}
If $f$ is pseudo-dislocated equi-continuous for all $x \in X$, then $f$ is called \textit{pseudo-dislocated equi-continuous} and the collection of all pseudo-dislocated equi-continuous functions from $X$ to $Y$ is denoted by $_pEC(X,Y)$  respectively.

\begin{Def}
Let $\{x_n\}_{n \in D}$ be a net in a pseudo-dislocated uniform space $(X, \mathcal{U}_A)$ . Then $\{x_n\}_{n \in D}$ is said to be \textit{convergent} to $x\in X$, if for each entourage $U \in \mathcal{U}_A$, there exists an $m \in D$, such that $(x, x_n)\in U$ for all $n \geq m$.
\end{Def}

We can show that the following net convergence criteria result holds good for pseudo-dislocated equi-continuous mappings
\begin{thm}
Let $(X, \tau)$ and $(Y, \mathcal{U}_A)$ be a topological space and a pseudo-dislocated uniform space respectively.  Then a function $f\,:(X,\tau) \to (Y, \mathcal{U}_A)$ is pseudo-dislocated equi-continuous at $x \in X$ if and only if whenever a net $\{x_n\}_{n \in D}$ converges to $x$ in $X$, its  image net $\{f(x_n)\}_{n \in D}$ converges to $f(x)$ in $Y$.
\end{thm}

\section{Topologies on $EC(Y,Z)$ }
In this section, we introduce few topologies on  $EC(Y,Z)$ and $_pEC(Y,Z)$.\\
Let $(Y,\tau)$ and $(Z, \mathcal{U})$ be a topological space and a uniform space respectively. Then for $z\in Z$, $V \in \tau$ and $U \in \mathcal{U}$, we define:
\begin{center}
$(V,U)_z=\{f \in EC(Y,Z)\,|\, f(V) \subseteq U[z]\,\,\}$
\end{center}
Let $\mathcal{S}_{\tau, \mathcal{U}}=\{(V,U)_z\,|\,z\in Z,\,\,V \in \tau\text{ and } U \in \mathcal{U}\}$.
\begin{lem} \label{lem:3.1}
$\mathcal{S}_{\tau, \mathcal{U}}$ is a subbasis for a topology on $EC(Y,Z)$.
\end{lem}
\begin{proof}
Let $f \in EC(Y,Z)$. Then for $y \in Y$  and $U \in \mathcal{U}$, there exists some open neighbourhood $V_0$ of $y\in Y$ such that $f(V_0) \subseteq U[f(y)]$. Consider $f(y)=z\in Z$. Then we have  $f \in (V_0,U)_z$. Therefore  $EC(Y,Z)\subseteq \bigcup \mathcal{S}_{\tau, \mathcal{U}}$.
\end{proof}
The topology generated by this subbasis will be called the \textit{open-entourage topology} for $EC(Y,Z).$\\

Similarly, for $y \in Y$, $V\in \tau$ and $U \in \mathcal{U}$, let us  consider  $(V,U)_y =\{f \in EC(Y,Z)\,|\, f(V) \subseteq U[f(y)]\,\,  \}$. Then it can be shown that the  collection $\{(V,U)_y\,|\,y\in Y, V \in \tau\text{ and } U \in \mathcal{U}\}$ also forms a subbasis for a topology on $EC(Y,Z)$.\\
The topology generated by this subbasis is called the \textit{open-entourage topology of Type-I} for $EC(Y,Z)$.\\

Clearly, the open-entourage topology is finer than the open-entourage  topology of Type-I.\\
Similarly, let $(Y,\tau)$ and $(Z, \mathcal{U})$ be a topological space and a uniform space respectively and $y\in Y$, $z \in Z$ and for  $U \in \mathcal{U}$. 
\\ We define:
\begin{center}
$(y,U)_z=\{f \in EC(Y,Z)\,|\, f(y) \in U[z]\,\,\}$.
\end{center}
Let $\mathcal{S}_{\tau, \mathcal{U}}^{p}=\{(y,U)_z\,|\,z\in Z,\,\,y \in Y \text{ and } U \in \mathcal{U}\}$.
\begin{lem} 
$\mathcal{S}_{\tau, \mathcal{U}}^p$ is a subbasis for a topology on $EC(Y,Z)$.
\end{lem}
\begin{proof}
Let $f \in EC(Y,Z)$, then for $y \in Y$  and for $U \in \mathcal{U}$ we have $\left(f(y),f(y)\right)\in U$. Consider $f(y)=z\in Z$, then we have  $f(y) \in U[z]$. Hence $f \in (y,U)_z$ and therefore  $EC(Y,Z)\subseteq \bigcup \mathcal{S}_{\tau, \mathcal{U}}^p$.
\end{proof}
The topology generated by this subbasis will be called the \textit{point-entourage topology} for $EC(Y,Z).$
\\
Now, for $y\in Y$, $z \in Z$. Let  $U_t \in \mathcal{U}$ be a transitive entourage of $\mathcal{U}$, that is  $U_t \circ U_t \subset U_t$.\\ We define:\begin{center}
$(y,U_t)_z=\{f \in EC(Y,Z)\,|\, f(y) \in U_t[z]\,\,\}$.
\end{center}
For each uniform space $(Z,\mathcal{U})$, we have $Z \times Z \in \mathcal{U}$. Then $U_t = Z \times Z$ satisfies the property $U_t \circ U_t \subseteq U_t$. Therefore there always exists entourages of the type  $U_t\in \mathcal{U}$.\\
Let $\mathcal{S}_{\tau, \mathcal{U}}^{t,p}=\{(y,U_t)_z\,|\,z\in Z,\,\,y \in Y,  U_t \in \mathcal{U} \text{ such that } U_t \circ U_t\subseteq U_t\}$.\\
It may be verified that $\mathcal{S}_{\tau, \mathcal{U}}^{t,p}$ is a subbasis for a topology on $EC(Y,Z)$.
\\
The topology generated by this subbasis will be called the \textit{point-transitive-entourage topology} for $EC(Y,Z).$
\\
Now, we introduce a  topological structure on  the class of pseudo-dislocated equi-continuous function $_pEC(Y,Z)$.

Let $(Y,\tau)$ and $(Z, \mathcal{U}_A)$ be a topological space and a pseudo-dislocated uniform space respectively. Then for $z\in Z$, $V \in \tau$ and $U \in \mathcal{U}_A$, we define
\begin{center}
$(V,U)_z = \{f\in _pEC(Y,Z)\,|\,f(V)\subseteq U[z]\}$
\end{center}
Let $\mathcal{S}(\tau, \mathcal{U}_A)=\{(V,U)_z\,|\,z \in Z, V \in \tau \text{ and }U \in \mathcal{U}_A\}$.
\begin{lem}
$\mathcal{S}(\tau, \mathcal{U}_A)$ is a subbasis for a topology on $_pEC(Y,Z)$.
\end{lem}
\begin{proof}
Similar to Lemma \ref{lem:3.1}.
\end{proof}
The topology generated by this subbasis will be called the \textit{open-dislocated-entourage} for $_pEC(Y,Z)$.\\

In the following section, we provide investigations of the function spaces on $EC(Y,Z)$. The development for $_pEC(Y,Z)$,  being similar, is not shown in the paper to avoid repetition.
\section{Admissibility and Splittingness on $EC(Y,Z)$}
In this section, we introduce few topologies on  $EC(Y,Z)$ and investigate some of their properties. Admissibility and splittingness for such spaces are defined and their characterizations are also provided in this section.\\
\begin{Def}
Let $(Y, \tau)$ and $(Z, \mathcal{U})$ be a topological space and a uniform space respectively. Let $(X, \mu)$ be another topological space. Then for a map $g\,: X \times Y \to Z$, we define a map $g^*\,:X \to EC(Y,Z)$ by $g^*(x)(y)=g(x,y)$.
\end{Def} 
These mappings $g$ and $g^*$ are called \textit{associated maps}.
\begin{Def}
Let $(Y, \tau)$ and $(Z, \mathcal{U})$ be a topological space and a uniform space respectively. A topology $\mathfrak{T}$ on $EC(Y,Z)$ is called
\begin{enumerate}[$(i)$]
\item \textit{admissible} if the evaluation map $e\,:EC(Y,Z) \times Y \to Z$ defined by $e(f,y)=f(y)$ is equi-continuous.
\item \textit{splitting} if for each topological space $(X, \mu)$, equi-continuity of the map $g\,: X \times Y \to Z$ implies  continuity of the map $g^*\,:X \to EC(Y,Z)$, where $g^*$ is the associated map of $g$.
\end{enumerate}
\end{Def}
The following results show that equi-continuity at times behaves like continuity only.
\begin{prop}
Let $(X, \tau)$ and $(Y, \mu)$ be two topological spaces and $(Z, \mathcal{U})$ be a uniform space. Let $f\,:X \to Y$ and $g\,:Y \to Z$ be continuous and equi-continuous  functions  at $x \in X$ and $f(x) \in Y$ respectively. Then the composition map $g \circ  f\,: X \to Z$ is equi-continuous at $x \in X$.
\end{prop}
\begin{proof}
Let $U \in \mathcal{U}$ be any entourage in $\mathcal{U}$. Since the map $g$ is equi-continuous at $f(x)$, therefore there exists an open neighbourhood $V$ of $f(x)$ in $Y$, such that $g(V)\subseteq U[g(f(x))]$. We have $f(x)\in V$ and $f$ is continuous at $x$, thus there exists an open neighbourhood $W$ of $x$ in $X$ with $f(W)\subseteq V$. Hence, we have $g(f(W))\subseteq g(V) \subseteq U[g(f(x))]$, that is, $(g\circ f)(W) \subseteq U[(g\circ f)(x)]$. Therefore the composition map $g\circ f$ is equi-continuous at $x$. 
\end{proof}
In the light of the above result, now we provide a characterization of admissibility.
\begin{thm}\label{thm: a}
Let $(Y, \tau)$ and $(Z, \mathcal{U})$ be a topological space  and a uniform space respectively. Let $(X, \mu)$ be any topological space. Then a topology $\mathfrak{T}$ on $EC(Y,Z)$ is admissible if and only if continuity of the map $g^*\,:X \to EC(Y,Z)$ implies equi-continuity of the map $g\,:X \times Y \to Z$, where $g^*$ and $g$ are the associated maps.
\end{thm}
\begin{proof}
Let the topology $\mathfrak{T}$ on $EC(Y,Z)$ be admissible, that is, the evaluation map $e\,:EC(Y,Z)\times
Y \to Z$ be equi-continuous. Let $g^*\,:X \to EC(Y,Z)$ be any continuous map. We have to show that its associated map $g$ is equi-continuous. Since the map $g^*$ is continuous, therefore the map $h\,:X \times Y \to EC(Y,Z) \times Y$, defined by $h(x,y) = (g^*(x),y)$ is also continuous. Hence, by the last proposition, the composition map $e \circ h$ is equi-continuous. Now, for $(x, y) \in X \times Y$,  consider $e\circ h(x,y) = e(h(x,y))= e(g^*(x),y)=g^*(x)(y)=g(x,y)$. Hence $e \circ h \equiv g$. Therefore the map $g$ is equi-continuous.\\
Conversely, let the condition hold. Consider $X = EC(Y,Z)$ with the topology $\mathfrak{T}$. We define $g^*\,: EC(Y,Z)\to EC(Y,Z)$ as the identity map. Hence $g^*$ is continuous. Thus by the given hypothesis, its associated map $g\, : EC(Y,Z) \times Y \to Z$ is also equi-continuous. For any $(f,y) \in EC(Y,Z) \times Y$, consider $g(f,y)=g^*(f)(y)=f(y)=e(f,y)$, where $e$ is the evaluation map. Therefore $g \equiv e$ and hence equi-continuous. Thus the topology $\mathfrak{T}$ on $EC(Y,Z)$ is admissible.
\end{proof}
In the next set of theorems, we provide characterizations of admissibility and splittingness of the topologies on $EC(Y,Z)$ using net theory. We extend the concept of continuous convergence of continuous mappings \cite{An} for this purpose. But before that we quote a result about directed sets, which we shall use  in our proof.\\
Let $\Delta$ be a directed set. We add a point $\infty$ to $\Delta$ satisfying $\infty\geq n$ for all $n \in \Delta$ and write $\Delta_0=\Delta\cup \{\infty\}$. A topology $\mathcal{T}_0$ may be generated on $\Delta_0$ by declaring every singleton of $\Delta$ as open and neighbourhoods of $\infty$ being of the form $U_{n_0}=\{n\,\,:\,\,n \ge n_0\}$,  $n_0 \in \Delta$.
\begin{lem}\cite{s}
Let $(Y, \tau)$ be a topological space and $\{y_n\}_{n \in D}$ be a net in $Y$. Then the net $\{y_n\}_{n \in \Delta}$ converges to $y$ in $Y$ if and only if the function $f\,:\Delta_0 \to Y$ defined by $f(n)=y_n$ for $n \in \Delta$ and $f(\infty)=y$ is continuous at $\infty$.
\end{lem}
From this lemma we have the following remark:
\begin{rmk}
Let $(Y, \tau)$ be a topological space and $\{y_n\}_{n \in D}$ be net in $Y$. Then the net $\{y_n\}_{n \in \Delta}$ converges to $y$ in $Y$ if and only if the function $f\,:\Delta_0 \to Y$ defined by $f(n)=y_n$ for $n \in \Delta$ and $f(\infty)=y$ is continuous.
\end{rmk}
Now we come to our main results of this section.
\begin{Def}
Let $\{f_n\}_{n \in \Delta}$ be a net in $EC(Y,Z)$. Then $\{f_n\}_{n\in \Delta}$ is said to \textit{equi-continuously converge} to $f \in EC(Y,Z)$ if for each net $\{y_m\}_{m \in \sigma}$ in $Y$ converging to $y$, $\{f_n(y_m)\}_{(n,m)\in \Delta \times \sigma}$ converges to $f(y)$ in $Z$.
\end{Def}
\begin{thm}
Let $(Y, \tau)$ and $(Z, \mathcal{U})$ be a topological space  and a uniform space respectively. Let $(X, \mu)$ be any topological space. Then a topology $\mathfrak{T}$ on $EC(Y,Z)$ is splitting if and only if for each net $\{f_n\}_{n \in \Delta}$ in $EC(Y,Z)$, equi-continuous convergence of $\{f_n\}_{n\in \Delta}$ to $f$ implies that $\{f_n\}_{n \in \Delta}$ converges to $f$ under $\mathfrak{T}$.
\end{thm}
\begin{proof}
Let $\mathfrak{T}$ be splitting and $\{f_n\}_{n \in\Delta}$ equi-continuously converge to $f$. Let $\Delta_0=\Delta\cup \{\infty\}$ be equipped with the  topology as described after Theorem $\ref{thm: a}$. Define $g$ $:$ $\Delta_0 \times Y \to Z$ by $g(n,y) = f_n (y)$ for all $n \in \Delta$ and $g(\infty, y) =f(y)$. We show that the map $g$ is equi-continuous.  Now, the only non-constant convergent net in $\Delta_0$ is $\{n\}_{n \in \Delta}$ which converges to $\infty$. Hence if $\mathcal{S}$ is a convergent net in $\Delta_0 \times Y$, then $\mathcal{S} = S_1 \times S_2$, where $S_1 = \{n\}$ and $S_2 = \{y_m\}_{m \in \sigma}$, where $\{y_m\}_{m \in \sigma}$ is any convergent net in $Y$, which converges to some $y$ in $Y$. Then $\mathcal{S}$ converges to $\{\infty\} \times \{y\}$ for some $y \in Y$ and $g(\mathcal{S}) = \{f_n(y_m)\}_{(n,m)\in \Delta\times \sigma}$. By 	equi-continuous convergence of $\{f_n\}_{n \in \Delta}$, $g(\mathcal{S})$ converges to $f(y) = g(\infty, y)$. Hence, by the net theoretic characterization of equi-continuity, $g$ is equi-continuous at $(\infty,y)$. Now, consider any $(n ,y)\in \Delta \times Y$, and let $U$ be any entourage in $\mathcal{U}$. We have, $g(n,y)=f_n(y)$. Since $U\in \mathcal{U}$ and $f_n$ is equi-continuous,  there exists an open neighbourhood $V$ of $y$ such that $f_n(V) \subseteq U[f_n(y)]$. Thus, we get an open neighbourhood $\{n\} \times V$ of $(n,y)$ such that $g(\{n\}\times V)=f_n(V)\subseteq U[f_n(y)]$. That is, $g(\{n\} \times V) \subseteq U[g(n,y)]$. Therefore the map $g$ is equi-continuous at $(n,y)$, for all $(n,y) \in \Delta\times Y$.   As $\mathfrak{T}$ is splitting, this implies that the associated map $g^*\,:\Delta_0 \to EC(Y,Z)$ is continuous. Since $\{n\}_{n \in \Delta}$ converges to $\infty$ in $\Delta_0$, we have, $\{g^*(n)\}_{n \in \Delta}$ converges to $g^*(\infty)$. Now $g^*(n)(y) = g(n,y) = f_n(y)$ and $g^*(\infty)(y) = g(\infty, y) = f(y)$. That is, $g^*(n) = f_n$, $g^*(\infty) = f$. Hence $\{f_n\}_{n \in \Delta}$ converges to $f$ in $EC(Y,Z)$.

Conversely, suppose  equi-continuous convergence implies convergence. Let $g$ $:$ $X \times Y \to Z$ be equi-continuous. We need to show that its associated map $g^*$ is continuous. Let $\{x_n\}_{n \in \Delta}$ be any convergent net in $X$ which converges to $x \in X$. We have to show that the image net $\{g^*(x_n)\}_{n \in \Delta}$ converges to $g^*(x)$ in $EC(Y,Z)$. We define,  $g^*(x_n) =f_n$ and $g^*(x) = f$. Now, we show that $\{f_n\}_{n \in \Delta}$ converges to $f$ in $EC(Y,Z)$. This follows if the net $\{f_n\}_{n \in \Delta}$ equi-continuously converges to $f$. Let us consider, a net $\{y_m\}_{m \in \sigma}$ in $Y$ which converges to some $y$ in $Y$. Then $\{(x_n,y_m)\}_{(n,m)\in \Delta\times \sigma}$ converges to $(x,y)$ in $X \times Y$. As $g$ is equi-continuous, the image net $\{g(x_n,y_m)\}_{(n,m)\in \Delta\times \sigma}$ converges to $g(x,y)$ in $Z$. But $g(x_n,y_m) =g^*(x_n)(y_m)=f_{n}(y_m)$ and $g(x,y) =g^*(x)(y)=f(y)$. That is, $\{f_n(y_m)\}_{(n,m)\in \Delta\times \sigma}$ converges to $f(y)$ in $Z$. Hence $\{f_n\}_{n \in \Delta}$ equi-continuously converges to $f$ in $EC(Y,Z)$. Thus by the hypothesis, we have $\{f_n\}_{n \in \Delta}$ converges to $f$ in $EC(Y,Z)$. That is, $\{g^*(x_n)\}_{n \in \Delta}$ converges to $g^*(x)$ in $EC(Y,Z)$. Hence $g^*$ is continuous. Therefore, $\mathfrak{T}$ is splitting.
\end{proof}
On a similar line,  characterization of admissibility is also provided below.
\begin{thm} \label{thm:3.10}
Let $(Y, \tau)$ and $(Z, \mathcal{U})$ be a topological space and a uniform space respectively. Let $(X, \mu)$ be any topological space. Then a topology $\mathfrak{T}$ on $EC(Y,Z)$ is admissible if and only if for each net $\{f_n\}_{n \in \Delta}$ in $EC(Y,Z)$, convergence of $\{f_n\}_{n\in \Delta}$ to $f$ in $EC(Y,Z)$ implies equi-continuous convergence of  $\{f_n\}_{n \in \Delta}$ to $f$.
\end{thm}
\begin{proof}
Let $\mathfrak{T}$ be admissible and $\{f_n\}_{n \in \Delta}$ be any net in $EC(Y,Z)$ such that $\{f_n\}_{n \in\Delta}$ converges to $f$. Let us define $ g^* : \Delta_0 \to EC(Y,Z)$ as $g^*(n) = f_n$ and $g^*(\infty) = f$, where $\Delta_0$ is generated by $\Delta$. Now the only non constant convergent net in $\Delta_0$ is $\{n\}$  which converges to $\infty$ and $\{g^*(n)\}_{n \in \Delta} = f_n$ converges to $f = g^*(\infty)$, by the given hypothesis. Hence $g^*$ is continuous. Therefore the associated map $g : \Delta_0 \times Y \to Z$ is equi-continuous. Let $\{y_m\}_{m \in \sigma}$ be any net  in $Y$ such that $\{y_m\}_{m \in \sigma}$ converges to $y$ in $Y$. Then $\{(n, y_m)\}_{(n,m)\in \Delta\times \sigma}$ is a convergent net in $\Delta_0 \times Y$ which converges to $(\infty, y)$. Therefore $\{g(n, y_m)\}_{(n,m)\in \Delta\times \sigma}$ converges to $g(\infty, y)$. That is, $\{g^*(n)(y_m)\}_{(n,m)\in \Delta\times \sigma}$ converges to $g^*(\infty)(y)$, which implies $\{f_n(y_m)\}_{(n,m)\in \Delta\times \sigma}$ converges to $f(y)$. Hence $\{f_n\}_{n \in \Delta}$ equi-continuously converges to $f$.    

Conversely, let $g^*$ be continuous. We have to show that its associated map $g$ is equi-continuous. Let $\{x_n\}_{n \in \Delta}$ and $\{y_m\}_{m \in \sigma}$ be two convergent nets in $X$ and $Y$ respectively  such that $\{(x_n, y_m)\}_{(n,m)\in \Delta\times \sigma}$ converges to $(x,y)$. Since $\{x_n\}_{n \in \Delta}$ converges to $x$ and $g^*$ is continuous, therefore the image net $\{g^*(x_n)\}_{n \in \Delta}$ converges to $g^*(x)$. Let us define $g^*(x_n) = f_{x_n}$ and $g^*(x) = f_x$. Then, we have $\{f_{x_n}\}_{n \in \Delta}$ converges to $f_x$ in $EC(Y,Z)$. Thus by the given hypothesis, $\{f_{x_n}\}_{n \in \Delta}$ equi-continuously converges to $f_x$. Then for the convergent net $\{y_m\}_{m \in \sigma}$ which converges to $y$, we have $\{f_{x_n}(y_m)\}_{(n,m)\in \Delta\times \sigma}$ converges to $f_x(y)$, that is $\{g(x_n , y_m)\}_{(n,m)\in \Delta\times \sigma}$ converges to $g(x,y)$. Hence $g$ is  equi-continuous. Therefore $\mathfrak{T}$ is admissible.
\end{proof}
Below, we mention a lemma without proof which is valid for  function spaces of continuous functions as well as of continuous multifunctions\cite{s}. Here $\mu \geq \tau$, means $\tau \subseteq \mu$.
\begin{lem}
Let $\tau$ and $\mu$ be two topologies on $EC(Y,Z)$ and $\mu \geq \tau$. Then admissibility of $\tau$ implies admissibility of $\mu$. On the other hand, if $\mu$ is splitting, then $\tau$ is also splitting.
\end{lem} 
\begin{proof}
Easy and left for the readers.
\end{proof}
Now we provide examples of admissible and splitting topologies using the results obtained  so far.\\
In our next pair of theorems, we show that  open-entourage topology is admissible whereas point-transitive-entourage topology is splitting.
\begin{thm}
Let $(Y, \tau)$ and $(Z, \mathcal{U})$ be a topological space and a uniform space respectively. Then the open-entourage topology on $EC(Y,Z)$ is admissible.
\end{thm}
\begin{proof}
Let $(Y, \tau)$ and $(Z, \mathcal{U})$ be a topological space and a uniform space respectively. We have to show that the open-entourage topology on $EC(Y,Z)$ is admissible, that is, for each net $\{f_n\}_{n \in \Delta}$ in $EC(Y,Z)$, convergence of $\{f_n\}_{n \in \Delta}$ to $f$ in $EC(Y,Z)$ implies equi-continuous convergence of $\{f_n\}_{n \in \Delta}$ to $f$.\\
Let $\{y_m\}_{m \in \sigma}$ be any convergent net in $Y$ which converges to $y$. We have to show that the net $\{f_n(y_m)\}_{(n,m) \in \Delta\times \sigma}$ converges to $f(y)$ in $(Z, \mathcal{U})$.\\
Let $U$ be any entourage in $\mathcal{U}$. Then there exists some $U_0 \in \mathcal{U}$ such that $U_0 \circ U_0 \subset U$. As $f$ is equi-continuous at $y_m$ and  $U_0 \in \mathcal{U}$, therefore there exists an open neighbourhood $V_0 \in \tau$ of $y_m$ such that $f(V_0)\subseteq U_0[f(y_m)]$, which implies $f \in (V_0,U_0)_{f(y_m)}$. Since the net $\{f_n\}_{n \in \Delta}$ converges to $f$ in $EC(Y,Z)$ and $(V_0,U_0)_{f(y_m)}$ is a subbasic open neighbourhood of $f$, therefore $f_n \in (V_0,U_0)_{f(y_m)}$ eventually. We have $f_n(V_0)\subseteq U_0[f(y_m)]$, whence  $f_n(y_m)\in U_0[f(y_m)]$ eventually. Hence we have $(f_n(y_m),f(y_m))\in U_0$ eventually.\\
Now, consider the net $\{y_m\}_{m \in \sigma}$ converging to $y$ in $Y$. As  $f \in EC(Y,Z)$,  the image net $\{f(y_m)\}_{m \in \sigma}$ converges to $f(y)$, that is, for $U_0^{-1} \in \mathcal{U}$, we have $(f(y),f(y_m))\in U_0^{-1}$, which implies $(f(y_m),f(y))\in U_0$ eventually.  Hence $(f_n(y_m),f(y_m))\circ (f(y_m),f(y)) \in U_0 \circ U_0 \subset U$ eventually. Thus we have $(f_n(y_m),f(y))\in U$ eventually and therefore the net $\{f_n(y_m)\}_{(n,m) \in \Delta\times\sigma}$ converges to $f(y)$ in $Z$. Therefore by Theorem \ref{thm:3.10}, the  open-entourage topology on $EC(Y,Z)$ is admissible.
\end{proof}
In the next theorem, we show that the point-transitive-entourage topology on $EC(Y,Z)$ is splitting.
\begin{thm}
Let $(Y, \tau)$ and $(Z, \mathcal{U})$ be a topological space and a uniform space respectively. Then the point-transitive-entourage topology on $EC(Y,Z)$ is splitting.
\end{thm}
\begin{proof}
Let $(Y, \tau)$ and $(Z, \mathcal{U})$ be a topological space and a uniform space respectively. We have to show that the point-transitive-entourage topology on $EC(Y,Z)$ is splitting, that is for each net $\{f_n\}_{n \in \Delta}$ in $EC(Y,Z)$, equi-continuous convergence to $\{f_n\}_{n \in \Delta}$ to $f$ implies  convergence of $\{f_n\}_{n \in \Delta}$ to $f$  in $EC(Y,Z)$.\\
Let $(y, U_t)_z$ be any subbasic open neighbourhood of $f$ in $EC(Y,Z)$. Then, $f(y)\in U_t[z]$, that is, $(f(y),z)\in U_t$.
Let $\{y_m\} = y$ for each $m \in \sigma$, be a constant net. Then $\{y_m\}_{m \in \sigma}$ converges to $y$ in $Y$. Since the net $\{f_n\}_{n \in \Delta}$ equi-continuously converges to $f$,  the net $\{f_n(y_m)\}_{(n,m) \in \Delta\times\sigma}$ converges to $f(y)$ in $Z$, that is, net $\{f_n(y)\}_{n \in \Delta}$ converges to $f(y)$. Then for $U_t \in \mathcal{U}$,  we have $U_t^{-1}\in \mathcal{U}$, which implies $(f(y),f_n(y)) \in U_t^{-1}$ eventually. Thus we have $(f_n(y),f(y))\in U_t$ eventually. Accordingly, we have  $(f_n(y),f(y))\circ (f(y),z) \in U_t \circ U_t \subset U_t$. Therefore $(f_n(y),z)\in U_t$ eventually which implies $f_n\in (y,U_t)_z$ eventually. Hence net $\{f_n\}_{n \in \Delta}$ converges to $f$ in $EC(Y,Z)$. Thus point-transitive-entourage topology on $EC(Y,Z)$ is splitting.
\end{proof}

\section{Dual Topology For Equi-Continuous Functions}
In this section, we introduce the notion of dual topology for the topologies on $EC(Y,Z)$. We provide here  interesting relationships regarding the splittingness and admissibility of a topology on equi-continuous functions and its dual.\\
For a topological space $(Y, \tau)$ and a uniform space $(Z, \mathcal{U})$, let $f \in EC(Y,Z)$, $U \in \mathcal{U}$ and $y \in Y$. Then by the definition of equi-continuity, there exists $V \in \tau$ of $y$ such that  $f(V) \subseteq U[f(y)]$. We denote the open set $V$ obtained this way by $U(f,y)$. Now we define :
\begin{center}
$\mathcal{O}_Z(Y)= \{U(f,y)\,:\, U \in \mathcal{U}\,\,, f\in EC(Y,Z), y \in Y\}$.
\end{center}
\begin{Def}
Let  $(Y, \tau)$ and $(Z, \mathcal{U})$ be a topological space and a uniform space respectively. Let $EC(Y,Z)$ be the set of all equi-continuous functions from $Y$ to $Z$. Then for subsets $\mathbb{H}\subseteq \mathcal{O}_Z(Y)$, $\mathcal{H}\subseteq EC(Y,Z)$ and $U \in \mathcal{U}$, we define:\\
$(\mathbb{H},U)= \{f\in EC(Y,Z)$ $|$ $U(f,y)\in \mathbb{H}$ for each $y \in Y$ $\}$\\
$(\mathcal{H},U)=\{U(f,y)$ $|$ $f\in \mathcal{H}$, \, $y \in Y$ $\}.$
\end{Def}
Let  $(Y, \tau)$ and $(Z, \mathcal{U})$ be a topological space and a uniform space respectively and $f \in EC(Y,Z)$, $U \in\mathcal{U}$. Then for each $y \in Y$, there exists $V \in \tau$ such that $f(V) \subseteq U[f(y)]$. Then $\mathbb{H} = \{U(f,y)$ $|$ $y \in Y\}$ is a subset of $\mathcal{O}_Z(Y)$, such that $f \in (\mathbb{H},U)$.   Therefore one can always define the sets of the form $(\mathbb{H},U)$ and $(\mathcal{H},U)$ which are non empty and well defined.
\begin{Def}
Let $(Y, \tau)$ and $(Z, \mathcal{U})$ be a topological space and a uniform space respectively. Let $\mathbb{T}$ be a topology on $\mathcal{O}_Z(Y)$. Then we define:\begin{center}
$\mathcal{S}(\mathbb{T}) =\{(\mathbb{H},U)$ $|$ $\mathbb{H}\in \mathbb{T}, U \in \mathcal{U}\}$.
\end{center}
\end{Def}
\begin{thm}
$\mathcal{S}(\mathbb{T})$ is a subbasis for a topology on $EC(Y,Z)$.
\end{thm}
\begin{proof}
Let $f \in EC(Y,Z)$. Then for  $y\in Y$, $U\in \mathcal{U}$, there exists  $V_y \in \tau$ such that $f(V_y)\subseteq U[f(y)]$. Consider $V_y = U(f,y)$. As $V_y \in \mathcal{O}_Z(Y)$ and $\mathbb{T}$ is a topology on $\mathcal{O}_Z(Y)$, therefore there exists an open set $\mathbb{H}_y$, such that $V_y = U(f,y) \in \mathbb{H}_y$. Let $\mathbb{H}=\displaystyle{\bigcup_{y \in Y}}\mathbb{H}_y$.  Then $f \in (\mathbb{H},U)$.  Hence $EC(Y,Z) = \bigcup \mathcal{S}(\mathbb{T})$. Therefore $\mathcal{S}(\mathbb{T})$ is a subbasis for a topology on $EC(Y,Z)$.
\end{proof}
Now, we provide a topology on $\mathcal{O}_Z(Y)$ using the topology on $EC(Y,Z)$.
\begin{thm}
Let $\mathfrak{T}$ be a topology on $EC(Y,Z)$. Then
\begin{center}
$\mathcal{S}(\mathfrak{T}) = \{(\mathcal{H},U)$ $|$ $\mathcal{H}\in \mathfrak{T}$, $U \in \mathcal{U}\}$
\end{center}
is a subbasis for a topology on $\mathcal{O}_Z(Y)$.
\end{thm}
\begin{proof}
Let $U(f,y) \in \mathcal{O}_Z(Y)$. Clearly $f \in EC(Y,Z)$ and hence $f \in \mathcal{H}$ for some $\mathcal{H}\in \mathfrak{T}$. Then  $U(f,y)\in (\mathcal{H},U)$. Therefore $\mathcal{O}_Z(Y) =\bigcup (\mathcal{H},U)$. Hence $\mathcal{S}(\mathfrak{T})$ is a subbasis for a topology on $\mathcal{O}_Z(Y)$.
\end{proof}
The topologies defined above on $EC(Y,Z)$ and $\mathcal{O}_Z(Y)$ are denoted by $\mathfrak{T}(\mathbb{T})$ and $\mathbb{T}(\mathfrak{T})$  respectively. We shall refer these topologies  as \textit{dual} to $\mathbb{T}$ and $\mathfrak{T}$ respectively.\\
Now we define the splittingness and admissibility on $\mathcal{O}_Z(Y)$ and investigate the possible relationships between a topology on $EC(Y,Z)$ and its dual and vice-versa.
\begin{Def}
Let $(X, \tau)$ and $(Y, \mu)$ be two topological spaces. A multifunction $F : X \to Y$ is called
\begin{enumerate}[$(i)$]
\item \textit{upper semi continuous} (or $u.s.c.$, in brief) at $x \in X$ if for each open set $V \subseteq Y$ with $F(x) \subseteq V$, there exists an open set $U$ of  $X$ such that $x \in U$ and $F(U) \subseteq V$;
\item \textit{lower semi continuous} (or $l.s.c$, in brief) at $x \in X$ if for each open set $V \subseteq Y$ with $F(x) \cap V \neq \emptyset$, there exists  an open set $U$ of  $X$ such that $x \in U$ and $F(u) \cap V \neq \emptyset$ for every $u \in U$;
\item \textit{continuous} at $x \in X$, if it is both $u.s.c.$ and $l.s.c.$ at $x$;
\item \textit{continuous} (resp. $u.s.c.$, $l.s.c.$) if it is continuous (resp. $u.s.c.$, $l.s.c.$) at each point of $X$.
\end{enumerate}
\end{Def}
\begin{Def}
Let  $(Y, \tau)$ and $(Z, \mathcal{U})$ be a topological space and a uniform space respectively. Let $(X, \mu)$ be another topological space. Let $g\,:X \times Y \to Z$ and $g^*\,:X \to EC(Y,Z)$ be two associated maps. . Then we define a multifunction $\overline{g}\,:X \times \mathcal{U} \to \mathcal{O}_Z(Y)$  by $\overline{g}(x, U)=\{ U(g^*(x),y)\,|\,y \in Y\}=\{U(g_x, y)\,|\,y \in Y\}$, for every $x \in X$ and $U \in \mathcal{U}$.
\end{Def}
\begin{Def}
Let  $(Y, \tau)$ and $(Z, \mathcal{U})$ be a topological space and a uniform space respectively. Let $(X, \mu)$ be another topological space. A multifunction $M\,:X \times \mathcal{U} \to \mathcal{O}_Z(Y)$ is called \textit{upper semi continuous with respect to the first variable} if the map $M_U\,:X \to \mathcal{O}_Z(Y)$ defined by $M_U(x) = M(x,U)$ is upper semi continuous for every $x \in X$ and for a fixed $U \in \mathcal{U}$.
\end{Def}
Now, we are in position to define the admissibility and splittingness  of the topological space $(\mathcal{O}_Z(Y),\mathbb{T} )$.
\begin{Def}
Let  $(Y, \tau)$ and $(Z, \mathcal{U})$ be a topological space and a uniform space respectively. Let $(X, \mu)$ be another topological space. Then topology $\mathbb{T}$ on $\mathcal{O}_Z(Y)$ is called
\begin{enumerate}[$(i)$]
\item \textit{splitting} if equi-continuity of the map $g\,:X \times Y \to Z$ implies upper semi continuity with respect to the first variable of the map $\overline{g}\,:X \times \mathcal{U} \to \mathcal{O}_Z(Y)$;
\item \textit{admissible} if for every map $g^*\,: X \to EC(Y,Z)$, upper semi continuity with respect to the first variable of the map $\overline{g}\,:X \times \mathcal{U} \to \mathcal{O}_Z(Y)$ implies equi-continuity of the map $g\,:X \times Y \to Z$. 
\end{enumerate}
\end{Def} 
In the remaining part of this section, we investigate  how duality links the admissibility and splittingness of a topology on $EC(Y,Z)$ and that on $\mathcal{O}_Z(Y)$.
\begin{thm}
A topology $\mathbb{T}$ on $\mathcal{O}_Z(Y)$ is splitting if and only if its dual topology $\mathfrak{T}(\mathbb{T})$ on $EC(Y,Z)$ is splitting.
\end{thm}
\begin{proof}
Let $(\mathcal{O}_Z(Y), \mathbb{T})$ be splitting, that is, for every topological space $(X, \mu)$, equi-continuity of the map $g\,:X \times Y \to Z$ implies upper semi continuity with respect to the first variable of the map $\overline{g}\,:X \times \mathcal{U}\to \mathcal{O}_Z(Y)$. We have to show that the topology $\mathfrak{T}(\mathbb{T})$ on $EC(Y,Z)$ is splitting, that is for every topological space $X$, equi-continuity of the map $f: X \times Y \to Z$ implies continuity of the associated map $f^*: X \to EC(Y,Z)$. Therefore, it is sufficient to show that upper semi continuity with respect to the first variable of the map $\overline{g} : X \times \mathcal{U} \to \mathcal{O}_Z(Y)$ implies continuity of the associated map $g^*: X \to EC(Y,Z)$.\\
Let $x \in X$ and $(\mathbb{H}, U)\in \mathfrak{T}(\mathbb{T})$ be a subbasic open neighbourhood of $g^*(x)$. Then $g^*(x)\in (\mathbb{H},U)$, which implies $U(g^*(x),y) \in \mathbb{H}$ for each $y \in Y$.  Therefore $U(g_x,y)\in \mathbb{H}$ and hence $\overline{g}(x,U)\subseteq \mathbb{H}$. Now $\overline{g}: X \times \mathcal{U} \to \mathcal{O}_Z(Y)$ is upper semi continuous with respect the first variable and $\mathbb{H}$ is an open neighbourhood of $\overline{g}_U(x)$. Hence there exists an open neighbourhood $V$ of $x$ such that $\overline{g}_U(V) \subseteq  \mathbb{H}$. Now, for an element $x' \in V$, we have $\overline{g}_U(x') \subseteq  \mathbb{H}$. Therefore $\overline{g}(x',U) \subseteq \mathbb{H}$ and hence $U(g_{x'}, y) \in \mathbb{H}$ for each $y \in Y$. That is, $U(g^*(x'), y)\in \mathbb{H}$ for every $x' \in V$, which implies $g^*(x')\in (\mathbb{H}, U)$ for every $x' \in V$. Thus $g^*(V)\subseteq (\mathbb{H},U)$. Therefore the map $g^*$ is continuous.\\
Conversely, let $\mathfrak{T}(\mathbb{T})$ be splitting, we have to show that the topology $\mathbb{T}$ is splitting. For this, it is sufficient to show that $\overline{g}:X \times \mathcal{U} \to \mathcal{O}_Z(Y)$ is upper semi continuous with respect to the first variable provided that the map $g^*\,: X \to EC(Y,Z)$ is continuous. Let, for a fixed $U \in \mathcal{U}$ and $x \in X$,   $\mathbb{H} \in \mathcal{O}_Z(Y)$ be an open neighbourhood of $\overline{g}(x,U)$. That is $\overline{g}(x,U) \subseteq \mathbb{H}$ which implies $U(g_x, y) \in \mathbb{H}$ for each $y \in Y$. Therefore  $U(g^*(x),y) \in \mathbb{H}$ for each $y \in Y$. Thus we have $g^*(x) \in (\mathbb{H}, U)$. Now the map $g^*$ is given to be continuous and $(\mathbb{H},U)$ is an open neighbourhood of $g^*(x)$. Thus there exists an open neighbourhood $V$ of $x$ such that $g^*(V) \subseteq (\mathbb{H},U)$. Now for, any $x' \in V$, we have $g^*(x') \in (\mathbb{H},U)$. Therefore, $U(g^*(x'),y) = U(g_{x'}, y) \in \mathbb{H}$ for every $x' \in V$. Hence, we have $\overline{g}_U(x') \subseteq \mathbb{H}$,  for all $x' \in V$. Hence $\overline{g}_U(V) \subseteq \mathbb{H}$. Hence the map $\overline{g}$ is upper semi continuous with respect to the first variable. Thus, the topology $\mathbb{T}$  is a splitting.
\end{proof}
\begin{thm}
A topology $\mathbb{T}$ on $\mathcal{O}_Z(Y)$ is admissible if and only if its dual topology $\mathfrak{T}(\mathbb{T})$ on $EC(Y,Z)$ is admissible.
\end{thm}
\begin{proof}
Let the topology $\mathbb{T}$ on $\mathcal{O}_Z(Y)$ be  admissible, that is, for every  topological space $(X, \mu)$ and for every map $g^*: X \to EC(Y,Z)$, upper semi continuity of the map $\overline{g}: X \times \mathcal{U} \to \mathcal{O}_Z(Y)$ with respect the first variable  implies equi-continuity  of the map $g: X \times Y \to Z$. We have to show that the topology $\mathfrak{T}(\mathbb{T})$ is admissible, that is continuity of $g^*: X \to EC(Y,Z)$ implies equi-continuity of its associated map $g\,:X \times Y \to Z$. Thus it is sufficient to prove that $\overline{g}: X \times \mathcal{U} \to \mathcal{O}_Z(Y)$ is upper semi continuous with respect to the first variable provided the map $g^*:X \to EC(Y,Z)$ is continuous.\\
Let us have, for fixed $U \in \mathcal{U}$ and $x \in X$,  a subbasic open neighbourhood $\mathbb{H}$ of $\overline{g}(x,U)$. Therefore $\overline{g}(x, U) \subseteq \mathbb{H}$. That is, $\overline{g}_U(x) \subseteq \mathbb{H}$ which implies $U(g_x, y)\in \mathbb{H}$ for each $y \in Y$.  
Thus $g^*(x) \in (\mathbb{H},U)$. Since the map $g^*$ is given to be  continuous and $(\mathbb{H},U)$ is a subbsaic open neighbourhood of $g^*(x)$, therefore there exists an open neighbourhood $V$ of $x$ such that $g^*(V) \subseteq (\mathbb{H},U)$. Now, for $x' \in V$, we have $g^*(x') \in (\mathbb{H},U)$, that is $U(g^*(x'),y)=U(g_{x'},y) \in \mathbb{H}$ for each $y \in Y$. Thus $\overline{g}_U(x') \subseteq \mathbb{H}$ for all $x' \in V$. Hence, $\overline{g}_U(V) \subseteq \mathbb{H}$. Therefore the map $\overline{g}$ is upper semi continuous with respect to the first variable. Hence the topology $\mathfrak{T}(\mathbb{T})$ is admissible.\\
Conversely, let $\mathfrak{T}(\mathbb{T})$ be admissible , we have to show that the topology $\mathbb{T}$ on $\mathcal{O}_Z(Y)$ is admissible. For this, it is sufficient to show that upper semi continuity with respect to the first variable of the map $\overline{g}:X \times \mathcal{U} \to \mathcal{O}_Z(Y)$ implies continuity  of the map $g^*:X \to EC(Y,Z)$.\\
Let $x \in X$ and $(\mathbb{H}, U)$ be a subbasic open neighbourhood of $g^*(x)$, that is $g^*(x) \in (\mathbb{H},U)$. Thus $U(g^*(x),y) \in \mathbb{H}$ for every $y \in Y$. Hence $\overline{g}_U(x) \subseteq \mathbb{H}$. Now the map $\overline{g}$ is given to be upper semi continuous with respect to the first variable and $\mathbb{H}$ is a subbaisc open neighbourhood of $ \overline{g}_U(x)$. Thus there exists an open neighbourhood $V$ of $x$ such that $\overline{g}_U(V) \subseteq \mathbb{H}$. Hence for $x' \in V$, we have $\overline{g}_U(x') \subseteq \mathbb{H}$,  which implies $\overline{g}(x',U)\subseteq \mathbb{H}$. Hence $U(g_{x'},y) = U(g^*(x'),y) \in \mathbb{H}$ for each $y \in Y$. Therefore $g^*(x') \in (\mathbb{H},U)$ for all $x' \in V$. Therefore $g^*(V) \subseteq (\mathbb{H},U)$. Thus  the topology $\mathbb{T}$ is admissible.
\end{proof}
In our next set of theorems, we investigate the relationship between a topology on $EC(Y,Z)$ and its dual.
\begin{thm}
A topology $\mathfrak{T}$  on $EC(Y,Z)$ is   splitting  if and only if its dual topology  $\mathbb{T}(\mathfrak{T})$ is  splitting.
\end{thm}
\begin{proof}
Let $\mathfrak{T}$ be a splitting  topology on $EC(Y,Z)$. We have to show that its dual topology $\mathbb{T}(\mathfrak{T})$  is also   splitting. For this, it is sufficient to prove that continuity of the map $g^*: X \to EC(Y,Z)$ implies upper semi continuity of the map $\overline{g}:X \times \mathcal{U} \to \mathcal{O}_Z(Y)$ with respect to the first variable.\\
Let $x \in X$ and $\mathcal{H} \in \mathfrak{T}$ be an open neighbourhood of $g^*(x)$. Then for any fixed $U \in \mathcal{U}$, $(\mathcal{H},U)\in \mathbb{T}(\mathfrak{T})$ is an open neighbourhood of $\overline{g}(x,U)$. That is, $\overline{g}(x,U) \subseteq (\mathcal{H},U)$. Now $\overline{g}(x, U) = \{U(g_x, y) \,|\,y\in Y\}\subseteq (\mathcal{H},U)$, hence $U(g_x,y)\in (\mathcal{H},U)$ for each $y \in Y$ by definition. This implies  $g^*(x) \in \mathcal{H}$. Since the map $g^*$ is given to be continuous and $\mathcal{H}$ is an open neighbourhood of $g^*(x)$, therefore there exists an open neighbourhood $V$ of $x$ such that $g^*(V) \subseteq \mathcal{H}$. Now, consider an element $x' \in V$, we have $g^*(x') \in \mathcal{H}$, that is $U(g_{x'},y) \in (\mathcal{H},U)$ for each $y \in Y$. Hence $\overline{g}(x',U)\subseteq (\mathcal{H},U)$, for all $x' \in V$. Therefore $\overline{g}_U( V) \subseteq (\mathcal{H},U)$ and the map $\overline{g}$ is upper semi continuous with respect to the first variable. Hence the result.\\
Conversely, let the topology    $\mathbb{T}(\mathfrak{T})$ be a splitting topology. We have to show that the topology $\mathfrak{T}$ on $EC(Y,Z)$ is splitting. It is equivalent to show that the map $g^*:X \to EC(Y,Z)$ is continuous provided the map $\overline{g}:X \times \mathcal{U} \to \mathcal{O}_Z(Y)$ is upper semi continuous.\\
Let $x \in X$ and $\mathcal{H}$ be an open neighbourhood of $g^*(x)$, that is, $g^*(x) \in \mathcal{H}$. For any fixed $U \in \mathcal{U}$, we have $U(g_x,y) \in (\mathcal{H},U)$ for each $y \in Y$. Therefore $\overline{g}(x,U)\subseteq (\mathcal{H},U)$ for a fixed $U \in \mathcal{U}$. Since the map $\overline{g}$ is given to be continuous with respect to the first variable, there exists an open neighbourgood $V$ of $x$ such that $ \overline{g}_U(V)\subseteq (\mathcal{H},U)$. Now, for $x' \in V$, we have $\overline{g}_U(x') \subseteq (\mathcal{H},U)$ which implies $U(g_{x'},y) \in (\mathcal{H},U)$ for each $y \in Y$. Therefore $g^*(x') \in \mathcal{H}$ for every $x' \in V$. That is, $g^*(V)\subseteq \mathcal{H}$. Hence the map $g^*$ is continuous.
\end{proof}
\begin{thm}
A topology $\mathfrak{T}$  on $EC(Y,Z)$ is   admissible  if and only if its dual topology  $\mathbb{T}(\mathfrak{T})$ is  admissible.
\end{thm}
\begin{proof}
Left for the reader.
\end{proof}
\begin{conc}
In this paper, we have studied topological structures on the family of equi-continuous mappings between a topological space and a uniform space. Important properties such as splittingness, admissibility etc. are introduced for such spaces and their characterizations are provided using net-theory. We have shown that similar studies can be carried out for pseudo-dislocated equi-continuous mappings also. It will be interesting to investigate the existence of the greatest splitting topology for such spaces. At the same time, the effect of duality on the existence of the greatest splitting topology needs to be investigated.
\end{conc}
\bibliographystyle{amsplain}

\end{document}